\documentclass[12pt,A4,leqno]{amsart}
\usepackage{amsfonts}
\usepackage{mathrsfs}
\usepackage[T1]{fontenc}
\usepackage{amssymb,amscd}
\usepackage{upgreek}
\usepackage{color}
\usepackage{epsf}
\usepackage{graphicx}
\usepackage{extarrows}
\usepackage[curve]{xypic}

\theoremstyle{plain}
\newtheorem{thm}{Theorem}[section]

\newtheorem{prop}[thm]{Proposition}
\newtheorem{cor}[thm]{Corollary}

\theoremstyle{definition}
\newtheorem{defi}[thm]{Definition}
\newtheorem{exam}[thm]{Example}
\newtheorem{rem}[thm]{Remark}

\newcommand{\R}{\mathbb R}
\newcommand{\Z}{\mathbb Z}

\newcommand{\nn}{\vskip 0.2cm}
\newcommand{\n}{\vskip 0.1cm}

\makeatletter
\renewcommand{\@secnumfont}{\bfseries}
\makeatother

\makeatletter
\def\section{%
  \@startsection{section}{1}
    {\z@}
    {2.0ex plus 0.8ex minus .1ex}
    {1.0ex plus .2ex}
    {\bfseries\large\centering\MakeUppercase}%
}
\makeatother

\begin{document}

\title [\ ] {On Riemannian polyhedra with non-obtuse dihedral angles in $3$-manifolds with positive scalar curvature}

\author{Li Yu}
\address{Department of Mathematics, Nanjing University, Nanjing, 210093, P.R.China.
  }
 \email{yuli@nju.edu.cn}


\keywords{Riemannian polyhedron, simple convex polytope, positive scalar curvature, Andreev's theorem}


\thanks{2020 \textit{Mathematics Subject Classification}.  51M20, 51F15, 52B10, 53C23, 57M50, 57S12\\
 This work is partially supported by 
 Natural Science Foundation of China (grant no.11871266) and 
 the PAPD (priority academic program development) of Jiangsu higher education institutions.}

\begin{abstract}
  We determine the combinatorial types of all the $3$-dimensional simple convex polytopes in $\R^3$ that can be realized as  mean curvature convex (or totally geodesic)  Riemannian polyhedra with non-obtuse dihedral angles in  Riemannian $3$-manifolds with positive scalar curvature.  This result can be considered as an analogue of Andreev's theorem on $3$-dimensional hyperbolic polyhedra with non-obtuse dihedral angles. In addition, we construct many examples of such kind of simple convex polytopes in higher dimensions.
  \end{abstract}

\maketitle

 \section{Introduction}

 A \emph{Riemannian polyhedron} of dimension $n$ is a  polyhedral domain $W$ with faces in an $n$-dimensional Riemannian manifold $(M,g)$ with the induced Riemannian metric $g(W)=g|_W$ (see Gromov~\cite[\S\,1]{Gromov14}). Moreover, each codimension-one face (or \emph{facet}) $F_i$ of $W$ is contained in a smooth hypersurface $\Sigma_i$ of $M$ such that
  \begin{itemize}
  \item  Whenever two facets $F_i$ and $F_j$ of $W$ are adjacent,
  the corresponding hypersurfaces $\Sigma_i$ and $\Sigma_j$ intersect transversely in $M$.\n
   
  \item The boundary $\partial F_i$ of each facet $F_i$
     consists of $F_i\cap F_j$ where $F_j$ ranges over all the facets adjacent to $F_i$, and the decomposition $\partial F_i= \bigcup_j F_i\cap F_j$ gives an 
     $(n-1)$-dimensional Riemannian polyhedron structure 
     to $F_i$.  \n
  \end{itemize}
  
  Note that the ambient manifold $M$ here
is not necessarily compact or closed.
Typical examples of Riemannian polyhedra are the intersections of finitely many domains with
smooth mutually transversal boundaries in a Riemannian
manifold (e.g. convex polytopes in the Euclidean space
$\R^n$).\n

 \begin{defi}
  Let $W_i$ be a Riemannian polyhedron in $(M_i,g_i)$, $i=1,2$. We call $W_1$ and $W_2$ \emph{combinatorially equivalent} if 
   there is a bijection between their faces that preserves
the inclusion relation. We call $W_1$ and $W_2$ \emph{pseudo-diffeomorphic} if there exists a homeomorphism 
$\varphi: W_1\rightarrow W_2$ which is the restriction of a diffeomorphism from an open neighborhood of $W_1$ in $M_1$ to an open neighborhood of $W_2$ in $M_2$, and we call 
 $\varphi$ a \emph{pseudo-diffeomorphism} from $W_1$ to $W_2$. If, moreover, $\varphi$ is face-preserving (i.e. $\varphi$ maps every face of $W_1$ homeomorphically onto a face of $W_2$), we call $\varphi$ a \emph{diffeomorphism} from $W_1$ to $W_2$. In this case,
we also say that
$W_2$ is a \emph{realization} of $W_1$ in $(M_2,g_2)$. 
It is clear that a diffeomorphism from $W_1$ to $W_2$ induces a combinatorial equivalence.
\end{defi}
   
By Wiemeler~\cite[Corollary 5.3]{Wiem13} (also see Davis~\cite[Corollary 1.3]{Da14}), two simple convex polytopes in $\R^n$ are diffeomorphic (as Riemannian polyhedra) if and only if they are combinatorially equivalent. Recall that a convex polytope $P$ in $\R^n$ is called \emph{simple} if every codimension-$k$ face of $P$ is contained in exactly $k$ different facets of $P$.
The reader is referred to~\cite{Zieg95} for the basic notions of 
convex polytopes. \n

 \begin{defi}
Let $W$ be a Riemannian polyhedron.
\begin{itemize}
\item We say that $W$ has \emph{acute} or \emph{non-obtuse dihedral angles} if the dihedral angle function on every codimension-two face of $W$ ranges in
$(0, \pi\slash 2)$ or $(0, \pi\slash 2]$. Especially, we call $W$ \emph{right-angled} if the dihedral angle function on every codimension-two face of $W$ is constantly $\pi\slash 2$.  \n

\item We call $W$ \emph{mean curvature convex} if every facet of $W$ has non-negative mean curvature in the ambient Riemannian manifold. Especially, we call $W$ \emph{totally geodesic} if every facet of $W$ is a totally geodesic submanifold.
\end{itemize}
\end{defi}

\noindent \textbf{Convention:} In this paper, the mean curvature of a boundary point of a domain in a Riemannian manifold is always taken with respect to the inward unit normal vector.\n

 The main purpose of this paper is to study the combinatorial types of all the $3$-dimensional simple convex polytopes in $\R^3$ that can be realized as mean curvature convex (or totally geodesic) Riemannian polyhedra with non-obtuse dihedral angles in Riemannian $3$-manifolds with positive scalar curvature. 
  \n
   
  For $n\geq 1$, let $\Delta^n$ and $[0,1]^n$ denote the standard $n$-simplex and $n$-cube in $\R^n$, respectively.
   The following is the main theorem of this paper.
     
   \begin{thm}\label{Thm-Main-1}
  Suppose $P$ is a $3$-dimensional simple convex polytope in $\R^3$.
  Then $P$ can be realized as a mean curvature convex Riemannian polyhedron 
   with non-obtuse dihedral angles in a Riemannian $3$-manifold with positive scalar curvature
    if and only if $P$ is combinatorially equivalent to a
    convex polytope that can be
     obtained from $\Delta^3$ by a sequence of
      vertex-cuts.
    \end{thm}
    
     \begin{defi}[Vertex-Cut] \label{Def-Vertex-Cut}
 Let $P$ be an $n$-dimensional simple convex polytope in $\R^n$ and $v$ a vertex of $P$.  Choose a plane
$H$ in $\R^n$ such that $H$ separates $v$ from the other vertices of $P$. Let $H_{\geq}$ and $H_{\leq}$ be the two
half spaces determined by $H$ and assume that $v$ belongs to $H_{\geq}$. Then $P\cap H_{\geq}$ is an $(n-1)$-simplex, and
$P\cap H_{\leq}$ is a simple convex polytope which we refer to as a \emph{vertex-cut} of $P$. 
For example, a vertex-cut of $\Delta^3$ is combinatorially equivalent to
 $\Delta^2\times [0,1]$ (the triangular prism).
 \end{defi}

   The simplicial polytope dual to a convex polytope that is obtained from $\Delta^3$ by a sequence of
      vertex-cuts is known as a stacked $3$-polytope. 
      By definition, a \emph{stacked $n$-polytope} is a polytope obtained from 
   $\Delta^n$ by repeatedly gluing another $n$-simplex onto one of its facets (see~\cite{MillReinStur07}).
 One reason for the significance of stacked polytopes is that, among all simplicial $n$-polytopes with a given number of vertices, the stacked polytopes have the fewest possible higher-dimensional faces. 
    \n

 By the proof of Theorem~\ref{Thm-Main-1}, we obtain
 the following corollary immediately.
  
      \begin{cor}\label{Cor-Main-2}
     Suppose $P$ is a $3$-dimensional simple convex polytope in $\R^3$. Then $P$ can be realized as a right-angled totally geodesic Riemannian polyhedron 
   in a Riemannian $3$-manifold with positive scalar curvature
    if and only if $P$ is combinatorially equivalent to a
    convex polytope that can be
     obtained from $\Delta^3$ by a sequence of
      vertex-cuts.
      \end{cor}
    
  Note that Theorem~\ref{Thm-Main-1} and Corollary~\ref{Cor-Main-2} still hold if we assume the scalar curvature of the ambient Riemannian $3$-manifold to be positive constant (see Corollary~\ref{Cor-Constant-SC}). 
\n

 Theorem~\ref{Thm-Main-1} can be thought of as an analogue of Andreev's theorem (see~\cite{Andreev70,Andreev70-2}) on totally geodesic polyhedra with non-obtuse dihedral angles in the $3$-dimensional hyperbolic space
   $\mathbb{H}^3$ (see~\cite{RoeHubDun07} for a new proof).
Andreev's theorem is essential
for proving Thurston’s Hyperbolization theorem for Haken $3$-manifolds. Especially,
Andreev's theorem tells us that a simple convex $3$-polytope $P$ can be realized as
a right-angled totally geodesic hyperbolic polyhedron in $\mathbb{H}^3$ if and only if $P$ has no prismatic $3$-circuits or prismatic $4$-circuits (this result was also obtained by A.~V.~Pogorelov in an earlier paper~\cite{Pog67}).
\n

 In addition, the following question proposed by M.~Gromov
 in~\cite{Gromov14} is related to our study.\n

 \textbf{Question} (Gromov~\cite[\S 1.7]{Gromov14}): What are the possible combinatorial types of mean curvature convex Riemannian polyhedra $W$ with acute dihedral angles in a Riemannian manifold with non-negative scalar curvature?\n

 If $W$ in the above question is the realization of $3$-dimensional simple convex polytope in a Riemannian $3$-manifold with positive scalar curvature, then
 the combinatorial type of $W$ must belong to the cases described in Theorem~\ref{Thm-Main-1}. But conversely, it is not clear whether we can construct a Riemannian polyhedron with acute dihedral angles for every combinatorial type described in Theorem~\ref{Thm-Main-1}.
 Indeed, the Riemannian polyhedra constructed in the proof of Theorem~\ref{Thm-Main-1} are all right-angled.  \n

The paper is organized as follows. In Section~\ref{Sec-Real-MAM}, we review the definition of real moment-angle manifold associated to a simple convex polytope. Besides, we quote a result from Wu-Yu~\cite{WuYu21} on when a $3$-dimensional real moment-angle manifold can admit a Riemannian metric with positive scalar curvature. In Section~\ref{Sec-Proof}, we give a proof of Theorem~\ref{Thm-Main-1} using the idea of ``doubling-smoothing'' of Riemannian manifolds
described in~\cite[\S 2.1]{Gromov14} along with the result from~\cite{WuYu21}. In Section~\ref{Sec-High-Dim}, we construct some examples of totally geodesic non-obtuse Riemannian polyhedra with positive scalar curvature in higher dimensions and propose a question.

\nn

\section{Real moment-angle manifolds}
\label{Sec-Real-MAM}

    Suppose $P$ is an $n$-dimensional simple convex polytope in Euclidean space $\R^n$.
         Let $\mathcal{F}(P)=\{F_1,\cdots, F_m\}$ be the set of all facets of $P$. Let $\Z_2=\Z\slash 2\Z$ and
    let $e_1,\cdots, e_m$ be a basis of
  $(\Z_2)^m$. Define a function $\lambda_0 : \mathcal{F}(P) \rightarrow (\Z_2)^m$ by
    \begin{equation}\label{equ-Lambda_0}
     \lambda_0(F_i) = e_i, \ 1\leq i \leq m.
    \end{equation}
     
      For any proper face $f$ of $P$, let $G_f$ denote the subgroup of $(\Z_2)^m$ generated by
  the set $\{ \lambda_0(F_i) \, |\, f\subset F_i \}$. For any point $p\in
  P$, let $f(p)$ denote the unique face of $P$ that contains $p$ in
  its relative interior. In~\cite[Construction 4.1]{DaJan91},
 the \emph{real moment-angle manifold} $\R \mathcal{Z}_{P}$ of $P$ is a closed orientable $n$-manifold
  defined by the following quotient construction
    \begin{equation} \label{Equ:Real-Moment-Angle}
     \R \mathcal{Z}_{P} := P\times (\Z_2)^m \slash \sim
    \end{equation}
  where $(p,g) \sim (p',g')$ if and only if $p=p'$ and $g^{-1}g' \in G_{f(p)}$.   
 So at every vertex of $P$, $2^n$ copies of $P$ are glued together so that locally they look like the $2^n$ cones of $\R^n$ bounded by the $n$ coordinate hyperplanes meeting at the origin. Let
 \begin{equation} \label{Equ-quotient-eta}
  \eta: P\times (\Z_2)^m \rightarrow  \R \mathcal{Z}_{P}
 \end{equation}
 be the quotient map.
   There is a \textit{canonical action} of
       $(\Z_2)^m$ on $\R\mathcal{Z}_{P}$ defined by
   $$ g'\cdot [(p,g)] = [(p,g'+g)],\ \forall\, p\in P, \, \forall\, g,g'\in (\Z_2)^m,  $$  
  whose orbit space can be identified with $P$.
    Let 
    $$\Theta_P: \R \mathcal{Z}_{P} \rightarrow P$$ 
    be the 
  orbit map.       
   Note that each facet $F$ of $P$ is also a simple convex polytope and, $\Theta^{-1}_P(F)$ is a disjoint union of several copies of $\R\mathcal{Z}_F$ embedded in $\R \mathcal{Z}_{P}$ with trivial normal bundles. \n

 The study of real moment-angle manifolds is an important subject in toric topology. The reader is referred to 
  Davis-Januszkiewicz~\cite{DaJan91}, Buchstaber-Panov~\cite{BP15}, Kuroki-Masuda-Yu~\cite{KurMasYu15} 
  and Wu-Yu~\cite{WuYu21} for more information of
 the topological and geometric properties of
  real moment-angle manifolds.
  The construction of $\R\mathcal{Z}_P$
  in~\eqref{Equ:Real-Moment-Angle} also makes sense for any smooth nice manifold with corners. The topology of such generalized spaces are studied in a recent paper Yu~\cite{Yu20}.\n

    In addition, we can realize $\R \mathcal{Z}_{P}$ as a non-degenerate intersection of $m-n$ real quadrics (quadric hypersurfaces) in $\R^m$, which induces a 
    $(\Z_2)^m$-invariant smooth structure on $\R\mathcal{Z}_P$ (see~\cite[\S 6]{BP15}). Consider a presentation of $P$ as follows:
    \begin{equation} \label{Equ-Presentation}
       P = P(A,b)=\{ x\in \R^n\,|\, \langle a_i,x\rangle + b_i\geq 0, \, i=1,\cdots,m\} 
    \end{equation}
   where $A=(a_1,\cdots,a_m)$ is an $n\times m$ real matrix. Since $P$ has a vertex, the rank of $A$ is equal to $n$. Define a map
   \begin{equation} \label{Equ-i-Ab}
     i_{A,b} : \R^n\rightarrow \R^m, \ \ i_{A,b}(x)=A^tx+b
     \end{equation}
   where $b=(b_1,\cdots,b_m)^t \in \R^m$. So the map
   $i_{A,b}$ embeds $P$ into the positive cone 
   $\R^m_{\geq 0}=\{ (x_1,\cdots,x_m)\in \R^m\,|\, x_i\geq 0, i=1,\cdots,m  \}$.
   We can define a space $\R\mathcal{Z}_{A,b}$ by the following commutative diagram
  \[   \xymatrix{
          \R\mathcal{Z}_{A,b} \ar[d] \ar[r]^{i_{\mathcal{Z}}}
                & \R^m \ar[d]^{\mu}  \\
          P  \ar[r]^{i_{A,b}} & \R^m_{\geq 0}
                 } 
                 \]     
     where $\mu(x_1,\cdots, x_m)=(x^2_1,\cdots, x^2_m)$.
     Clearly $(\Z_2)^m$ acts on
     $\R\mathcal{Z}_{A,b}$ with quotient space $P$ and
     $i_{\mathcal{Z}}$ is a $(\Z_2)^m$-equivariant embedding. It is easy to see that $\R\mathcal{Z}_{A,b}$ is homeomorphic to $\R\mathcal{Z}_P$. In addition, the image of $\R^n$ under $i_{A,b}$ is an affine plane of dimension $n$ in $\R^m$, which we can specify by $m-n$ linear equations:
    \begin{align*}
      i_{A,b}(\R^n) &=\{ y\in \R^m \,|\, y=A^tx+b, x\in\R^n\} \\
       &=\{y\in \R^m\,|\, \Gamma y =\Gamma b\}
       \end{align*}
       where $\Gamma=(\gamma_{jk})$ is an
       $(m-n)\times m$ matrix of rank $m-n$ so that
       $\Gamma A^t=0$. In other words, the rows of $\Gamma$ form a basis of all the linear relations
among $a_1,\cdots,a_m$. Then we can write the image of $\R\mathcal{Z}_{A,b}$ under $i_{\mathcal{Z}}$ explicitly as the common zeros of $m-n$ real quadratic equations in $\R^m$:
\begin{equation} \label{Equ-Quadrics}
  i_{\mathcal{Z}} (\R\mathcal{Z}_{A,b})= \big\{ 
  (y_1,\cdots,y_m)^t\in \R^m \,|\, \sum^m_{k=1} \gamma_{jk} y^2_k = \sum^m_{k=1} \gamma_{jk} b_k, \,
  1\leq j \leq m-n
\big\}. 
\end{equation}

 The above intersection of real quadrics is non-degenerate (i.e. the gradients of these quadrics are linearly independent everywhere in their intersection). This implies that $\R\mathcal{Z}_{A,b}$
 is embedded as an $n$-dimensional smooth submanifold in $\R^m$ where $(\Z_2)^m$ acts smoothly. So $P$ is embedded as a Riemannian polyhedron 
 in $i_{\mathcal{Z}}(\R\mathcal{Z}_{A,b})$ with the induced metric from
 $\R^m$. 
 
  \begin{exam} \label{Exam-Simplex}
    The standard simplex $\Delta^n \subset \R^n$ is defined
  by 
  $$\Delta^n =\{ (x_1,\cdots, x_n)\in \R^n\,|\, x_1+\cdots + x_n\leq 1, \, x_i\geq 0, \, i=1,\cdots, n.  \}$$
    By the notation in~\eqref{Equ-Presentation}, we have $\Delta^n= P(A,b)$ where 
  \[ A_{n\times (n+1)}=  \begin{pmatrix}
        1 & 0 & \cdots & 0 & -1   \\
        0 & 1 & \cdots & 0 & -1 \\
        \vdots & \vdots & & \vdots & \vdots \\
         0 & 0 & \cdots & 1 & -1
   \end{pmatrix}, \ \ \
    b= (0,0,\cdots,0,1)^t \in \R^{n+1}. \]
 So the image of the embedding $i_{A,b}:\Delta^n\hookrightarrow \R^{n+1}$ (see~\eqref{Equ-i-Ab}) is given by
  $$ \{ (x_1,\cdots, x_{n+1})\in \R^{n+1}\,|\, x_1+\cdots+x_{n+1}=1, \, x_i\geq 0,\, i=1,\cdots,n+1  \}.  $$
    Then by~\eqref{Equ-Quadrics},
  $i_{\mathcal{Z}}(\R\mathcal{Z}_{A,b}) \subset \R^{n+1}$ is given by the following equation 
  \[  y^2_1 + \cdots + y^2_{n+1}=1, \ (y_1,\cdots, y_{n+1})^t \in \R^{n+1} \]
  which is exactly the standard unit sphere $S^n$ in $\R^{n+1}$. Moreover, the canonical $(\Z_2)^{n+1}$-action on
  $\R\mathcal{Z}_{\Delta^n}$ is equivalent to the
   action of $(\Z_2)^{n+1}$ on $S^n$ by the reflections
   about the coordinate hyperplanes of $\R^{n+1}$. 
   More precisely, for each $1\leq i \leq n+1$, the $i$-th generator of $(\Z_2)^{n+1}$
   acts on $S^n$ by
   \[ (y_1,\cdots, y_{i-1}, y_i, y_{i+1}, \cdots , y_{n+1} ) \longrightarrow
    (y_1,\cdots, y_{i-1}, -y_i, y_{i+1}, \cdots , y_{n+1} ).  \]
    In the following, we think of $\R\mathcal{Z}_{\Delta^n}$ as $S^n$ equipped with the above $(\Z_2)^{n+1}$-action.   
  \end{exam}

    The theorem below tells us what kind of $3$-dimensional real moment-angle manifolds can admit a Riemannian metric with positive scalar curvature.
    
     \begin{thm}[Proposition 4.8 and Corollary 4.10 in Wu-Yu~\cite{WuYu21}] \label{Thm-WuYu21}
 \ \n
  Let $P$ be a $3$-dimensional simple convex polytope in $\R^3$ with $m$ facets. 
   Then $\R \mathcal{Z}_P$ admits a Riemannian metric with positive scalar curvature if and only if
  $P$ is combinatorially equivalent to a
    convex polytope that can be
     obtained from $\Delta^3$ by a sequence of
      vertex-cuts.
  Moreover, 
 we can choose the Riemannian metric with positive scalar curvature on $\R\mathcal{Z}_P$ to be invariant with respect to the canonical $(\Z_2)^m$-action on $\R\mathcal{Z}_P$.
\end{thm} 
   
 \n
 
   \section{Proof of Theorem~\ref{Thm-Main-1}} \label{Sec-Proof}
   
  Our proof of Theorem~\ref{Thm-Main-1} is inspired by Gromov's proof of the dihedral rigidity of the $n$-cube
  in~\cite{Gromov14}. Gromov's proof is
  based on the idea of doubling the cube $n$ times and uses the well-known fact that the $n$-dimensional torus admits no metric with positive scalar curvature.
  Here we observe that for a simple convex polytope $P$ with $m$ facets, doubling $P$
  $m$ times gives the real moment-angle manifold $\R\mathcal{Z}_P$. This relates the Riemannian metric on $P$ with certain boundary conditions to that on $\R\mathcal{Z}_P$ through the
  ``doubling-smoothing'' process described in~\cite{Gromov14}.
  Then we can prove Theorem~\ref{Thm-Main-1} by the result in Theorem~\ref{Thm-WuYu21}.
 But the description of the ``doubling-smoothing'' of Riemannian metrics in~\cite{Gromov14} is very sketchy, we will give more detailed exposition in our proof of Theorem~\ref{Thm-Main-1} below. Our proof slightly generalizes
 the argument in Gromov-Lawson~\cite[Theorem 5.7]{GromovLawson80}. \n
  
   Let $A$ be a subspace of a topological space $X$. The \emph{double of $X$ along $A$} is the quotient space of the disjoint union of two copies of $X$ by identifying each point of $A$ in one copy of $X$ to the same point in the other copy. \n

 Suppose $F$ and $F'$ are two facets of a Riemannian polyhedron $W$ which intersect transversely 
  at a codimension-two face $F\cap F'$.\n
 \begin{itemize}
 \item Let $\angle (F,F')_x$ denote the dihedral angle of $W$
  at a point $x\in F\cap F'$.\n
  \item Let $\angle (F,F')$ denote the dihedral angle of $W$
  at an arbitrary point of $F\cap F'$.
 \end{itemize}
  \nn

      \textbf{\textit{Proof of Theorem~\ref{Thm-Main-1}.}}
   \n
   
    Let $P$ be a simple convex $3$-polytope in $\R^3$
    whose facets are $F_1,\cdots, F_m$. Let $\Theta_P : \R\mathcal{Z}_P\rightarrow P$ be the orbit map of the canonical $(\Z_2)^m$-action on $ \R\mathcal{Z}_P$. So each $\Theta^{-1}_P(F_i)$, $1\leq i \leq m$, consists of some closed connected $2$-manifolds
   that intersect transversely in $\R\mathcal{Z}_P$.
   Since we are working in dimension $3$, we do not need to worry about the existence or uniqueness of the smooth structures in our manifolds. \n
   
   We first prove the ``if'' part. Suppose $P$ 
  is combinatorially equivalent to a convex polytope that can be obtained from $\Delta^3$ by a sequence of vertex-cuts. Then by Theorem~\ref{Thm-WuYu21}, there exists a 
  $(\Z_2)^m$-invariant Riemannian metric $g_0$ with positive scalar curvature on $\R\mathcal{Z}_P$.
  Note that $P$ is realized
  as the fundamental domain of the canonical $(\Z_2)^m$-action on $\R\mathcal{Z}_P$ which is bounded by the submanifolds $\Theta^{-1}_P(F_i)$, $1\leq i \leq m$.
 Moreover, by definition  each $\Theta^{-1}_P(F_i)$ is the fixed point set of the generator
 $e_i\in (\Z_2)^m$ under the canonical $(\Z_2)^m$-action
 (see~\eqref{equ-Lambda_0}).
 It is a standard fact that every connected component of the fixed point set (with the induced Riemannian metric) of an isometry on a Riemannian manifold is a totally geodesic submanifold (see~\cite[Theorem 1.10.15]{Kling95}). So each $\Theta^{-1}_P(F_i)$ consists of totally geodesic submanifolds of $(\R\mathcal{Z}_P,g_0)$. In addition, 
 the dihedral angles between any components of 
  $\Theta^{-1}_P(F_i)$ and $\Theta^{-1}_P(F_j)$ (whenever they intersect) are always equal to $\pi\slash 2$ since the 
  $(\Z_2)^m$-action on $(\R\mathcal{Z}_P,g_0)$ is isometric. Therefore, $P$ is realized as a
  right-angled totally geodesic Riemannian polyhedron
  in $(\R\mathcal{Z}_P,g_0)$. The ``if'' part is proved.
   \n
   
   Next, we prove the ``only if'' part.  
    We can think of $\R\mathcal{Z}_P$ as an iterated doubling of $P$ as follows.
 By the notation in the definition of $\R\mathcal{Z}_P$ (see~\eqref{equ-Lambda_0} and~\eqref{Equ-quotient-eta}), we define
  $$H_j :=\text{the subgroup of $\Z^n_2$ generated by $e_1,\cdots, e_j$},\ 1\leq j \leq m, \ \text{and} \ H_0:=\{0\}; $$ 
  $$ Y^{(j)} = \eta(P\times H_j),\ \ F^{(j)}_i=\eta(F_i\times H_{j}), \ 1\leq i,j \leq m. $$
  
 Then $Y^{(j)}$ is the gluing of $2^j$ copies of $P$ under $\eta$ whose boundary is
 $$ \partial Y^{(j)} =  \bigcup_{i > j} F^{(j)}_i . $$
 
   In addition, for any facet $F_i$ of $P$ and any element $g\in \Z^m_2$, let
   \[ F_{i,g}:= \eta(F_i\times \{g\}), \ 1\leq i \leq m. \]
   Then we have
    $$F^{(j)}_i = \bigcup_{g\in H_j} F_{i,g}, \ 1\leq i,j\leq m. $$
So $ Y^{(j)}$ is a Riemannian polyhedron in $\R\mathcal{Z}_P$ whose facets are $\{ F_{i,g} \,|\, i>j, g\in H_{j}\}$.
 \n

\begin{figure}
        \begin{equation*}
        \vcenter{
            \hbox{
                  \mbox{$\includegraphics[width=0.9\textwidth]{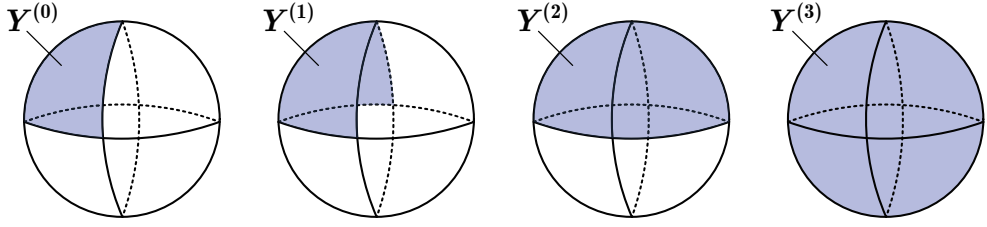}$}
                 }
           }
     \end{equation*}
   \caption{The filtration of $\R\mathcal{Z}_{\Delta^2} = S^2$
       } \label{p-Sphere}
   \end{figure}

   By identifying $P$ with $Y^{(0)}$, we have
  a filtration (see Figure~\ref{p-Sphere} for example)
   $$P= Y^{(0)} \subset Y^{(1)} \subset \cdots \subset Y^{(m)} = \R\mathcal{Z}_P.$$   
   Clearly, 
 $Y^{(j+1)}$ is
  the double of $Y^{(j)}$ along $F^{(j)}_{j+1}$ for each
  $0\leq j \leq m-1$. Note that $F^{(j)}_{j+1}$
  may not be connected. \n

 In the following, we do induction on $j$ and
assume that:
 \begin{itemize}
  \item[(a)] There exists a pseudo-diffeomorphism
$\varphi_j : Y^{(j)}\rightarrow W_j$ where $W_j$ is a mean curvature convex Riemannian polyhedron 
in a Riemannian $3$-manifold $(M_{j},g_{j})$ with  positive scalar curvature, and $\varphi_j$ maps every
 $F^{(j)}_i$ on $\partial Y^{(j)}$ to be a union of facets of 
$W_j$.
\end{itemize}
Note that here we cannot require $\varphi_j$ to be face-preserving when $j\geq 1$ since $W_j$ may have more facets than $Y^{(j)}$ (see Figure~\ref{p-Codimension-Two-Face} for example).

\n

 We say that a facet $E$ of $W_j$ \emph{comes from the facet} $F_i$ of $P$ if 
 $$\Theta_P(\varphi_j^{-1}(E))\subset F_i \ \text{or equivalently}\ E\subset \varphi_j(F^{(j)}_i)\ \text{where} \ j< i \leq m.$$
  There are two types of edges (codimension-two faces) on
$\partial W_j$ (see Figure~\ref{p-Codimension-Two-Face}):\n
 Type-I: The intersection of two facets that come from different facets of $P$.\n
 
Type-II: The intersection of two facets that come from the same facet of $P$.\n

 Moreover, we require $W_j$ to satisfy the following two conditions in our induction hypothesis. \n
 
 \begin{itemize}
 \item[(b)] The dihedral angles of $W_j$ are non-obtuse at
every Type-I edge on $\partial W_j$.\n
 
 \item[(c)] The dihedral angles of $W_j$ range in
 $(0,\pi]$ at every Type-II edge on $\partial W_j$.
 \end{itemize}
 
 \n

 \begin{figure}
        \begin{equation*}
        \vcenter{
            \hbox{
                  \mbox{$\includegraphics[width=0.94\textwidth]{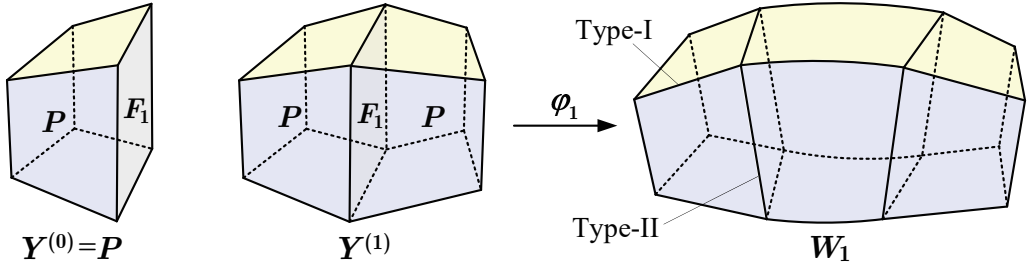}$}
                 }
           }
     \end{equation*}
   \caption{Two types of edges on the boundary of $W_1$
       } \label{p-Codimension-Two-Face}
   \end{figure}
 
\begin{figure}
        \begin{equation*}
        \vcenter{
            \hbox{
                  \mbox{$\includegraphics[width=0.69\textwidth]{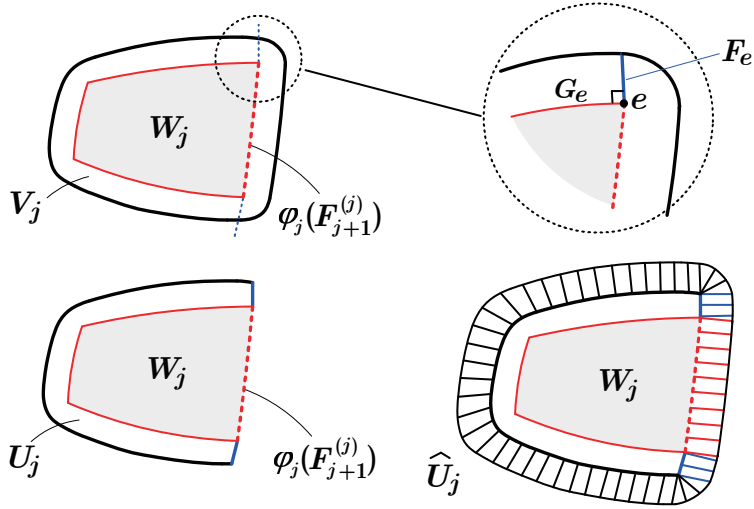}$}
                 }
           }
     \end{equation*}
   \caption{A domain $U_j$ containing $W_j$
       } \label{p-Neighborhood}
   \end{figure}
 
 From the above assumptions, we construct a pseudo-diffeomorphism 
 $\varphi_{j+1}$ from $Y^{(j+1)}$ to 
 some mean curvature convex Riemannian polyhedron 
in a Riemannian $3$-manifold with positive scalar curvature that satisfies (a),(b) and (c).
 First of all, define
   $$ V_j := \{ x\in M_j \,;\, \mathrm{dist}(x, W_j) \leq \delta  \},\ \delta \ll 1; $$
 
  By~\cite[Theorem 1.1]{Alem85}, we can multiply a smooth function $f:\R \rightarrow \R$ to 
 some components of the metric $g_j|_{W_{j}}$ and obtain a new metric $g'_j$ with positive scalar curvature
   on $W_j$
   so that the mean curvature of $g'_j$ at every facet of $W_j$ is positive. So without loss of generality, we can just assume that the mean curvature of $g_j$ at every facet of $W_j$ is positive at the beginning.
 Then for a sufficiently small $\delta$, it is not hard to show that $V_j$ is a mean curvature convex domain in $(M_j,g_j)$.\n

 Let $e$ be an edge on boundary of $\varphi_j ( F^{(j)}_{j+1} )$. We can write $e=G_e\cap \varphi_j ( F^{(j)}_{j+1})$ where $G_e$ is a facet of $W_j$.  
 Let $F_e$ be the union of the geodesic segments 
 in $V_j$ emanating from the points of $e$ that are 
 orthogonal to $G_e$ (see Figure~\ref{p-Neighborhood}).
 Then the dihedral angles between $F_e$ and the facets 
 in $\varphi_j ( F^{(j)}_{j+1})$ and $\partial V_j$ are all less than $\pi$. \n

 If we cut off an open subset from $V_j$ along $\varphi_j ( F^{(j)}_{j+1} )$ and
  these $F_e$'s,  we obtain a compact polyhedral domain $U_{j}$ containing $W_j$ where
$\partial U_j\cap 
  \partial W_j = \varphi_j ( F^{(j)}_{j+1} )$ (see Figure~\ref{p-Neighborhood}).
  Clearly, $U_j$ is also mean curvature convex.
  Moreover, define
   $$ \widehat{U}_j := \{ x\in M_j \,;\, \mathrm{dist}(x, U_j) \leq \delta  \},\ \delta \ll 1; $$
   \begin{align*}
    \widehat{W}_j := W_j \, \cup\, & \big\{\text{geodesic segments  in $\widehat{U}_j$ emanating orthogonally} \\
    &\text{\quad from faces in}\ \varphi_j ( F^{(j)}_{j+1} )\big\}.
    \end{align*}
    So $\widehat{U}_j\backslash U_j$ is a thin collar of
   $\partial \widehat{U}_j$.
   Then consider
   \begin{align} 
     D(U_j) &:= \big\{ p\in \widehat{U}_j\times [-1,1]\,;\, \mathrm{dist}\big(p, U_j \times \{ 0\}\big) =\varepsilon \big\}, \ 0<\varepsilon < \delta\slash 2, \label{Equ-D-Uj} \\
         D(W_j) &:= \big\{ p\in \widehat{W}_j\times [-1,1]\,;\, \mathrm{dist}\big(p, W_j \times \{ 0\}\big) =\varepsilon \big\}, \ 0<\varepsilon < \delta\slash 2. \label{Equ-D-Wj}
     \end{align}

   \begin{figure}
        \begin{equation*}
        \vcenter{
            \hbox{
                  \mbox{$\includegraphics[width=0.9\textwidth]{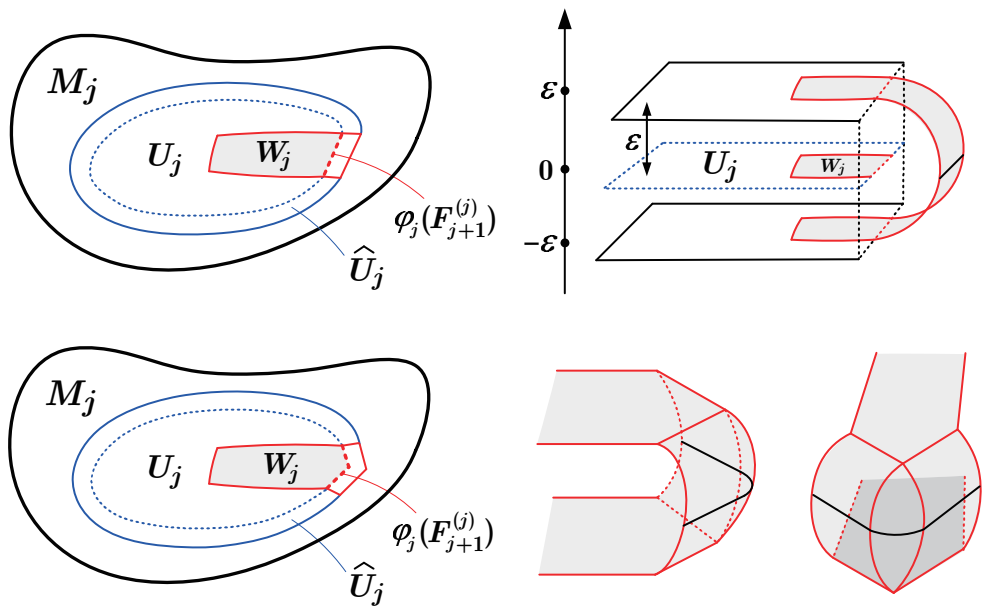}$}
                 }
           }
     \end{equation*}
   \caption{$D(U_j)$ and $D(W_j)$
       } \label{p-Double}
   \end{figure} 
       
   It is easy to see that $D(U_j)$ is homeomorphic to the double of $U_j$ along $\partial U_j$, and $D(W_j)$ is homeomorphic to the double of $W_j$ along $\varphi_j(F^{(j)}_{j+1})$ (see Figure~\ref{p-Double}). Hence $D(W_j)$ is homeomorphic
 to $Y^{(j+1)}$. More precisely, there exists a homeomorphism 
 $$\varphi_{j+1}: Y^{(j+1)}\rightarrow D(W_j)$$
  which maps a normal neighborhood $N(F^{(j+1)}_{j+1})$ of $F^{(j+1)}_{j+1}$ in $Y^{(j+1)}$ onto the bending region of $D(W_j)$ and maps the two components of
 $Y^{(j+1)}\backslash N(F^{(j+1)}_{j+1})$ onto the two copies of $W_j$ in $D(W_j)$ that are parallel to $W_j\times \{0\}$. Moreover, by our assumption that
 $\varphi_j: Y^{(j)}\rightarrow W_j$ is a pseudo-diffeomorphism, 
  we can extend $\varphi_{j+1}$ to a diffeomorphism from an open neighborhood of $Y^{(j+1)}$ in $\R\mathcal{Z}_P$ to an open neighborhood of $D(W_j)$ in $D(U_j)$. So $\varphi_{j+1}$ is also a pseudo-diffeomorphism.
   \n
 
 Here we remark that if the dihedral angles of $W_j$ were greater than $\pi$ at an edge in
   $\varphi_j(F^{(j)}_{j+1})$, then $D(U_j)$ and
   $D(W_j)$ would have some holes in the bending region. This is the reason why we need to assume $W_j$ to satisfy the condition (c). \n
   
   Geometrically, $D(U_j)$ and $D(W_j)$ will have some codimension-one creases where the induced metric from
   $\widehat{U}_j\times [0,1]$ is not smooth. In addition,   
   the facets in $\varphi_j(F^{(j)}_{j+1})\subset \partial W_j$ may have different types of local configurations which will cause different shapes of
  $D(W_j)$ in the bending region (see Figure~\ref{p-Double}). In particular in dimension $3$, there are possibly three types of local configurations of facets in $\varphi_j(F^{(j)}_{j+1})$ and $\partial U_j$ (see Figure~\ref{p-Configuration}). \n
  
     \begin{figure}
        \begin{equation*}
        \vcenter{
            \hbox{
                  \mbox{$\includegraphics[width=0.53\textwidth]{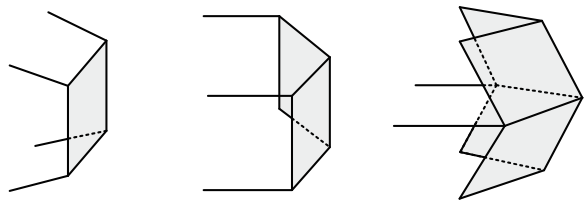}$}
                 }
           }
     \end{equation*}
   \caption{Local configurations of facets of $W_j$ in dimension $3$
       } \label{p-Configuration}
   \end{figure}

    Next, we use the (slightly generalized) argument in~\cite[Theorem 5.7]{GromovLawson80} to
  obtain a Riemannian metric on
   $D(U_j)$ with positive scalar curvature
    when $\varepsilon$ is sufficiently small. Let
     \begin{itemize}
     \item $ M_{j+1}= D(U_j)$ with the creases smoothed out;\n
     \item $g_{j+1}=$ the metric on $M_{j+1}$
     induced from the product metric $g_j \times g_{[-1,1]}$
     
      \quad\ \ \, on $U_j\times [-1,1]$.       
     \end{itemize}
     \n
     
 Using~\cite[Theorem 1.1]{Alem85} again, we can assume that
 the mean curvature of $g_j$ at every facet of $U_j$ is positive.   
 Then similarly to the proof of~\cite[Theorem 5.7]{GromovLawson80}, we can compute the scalar curvature of $g_{j+1}$ by estimating the principal curvatures of $D(U_j)$ in $U_j\times [-1,1]$ as follows. \n

  On the region parallel to $U_j\times \{0\}$ in $\widehat{U}_j\times [-1,1]$, the scalar curvature of  $g_{j+1}$ on $D(U_j)$ is clearly positive. So the difficulty comes from the bending region of $D(U_j)$. 
   The following argument is
  parallel to the argument in
   the proof of Gromov-Lawson~\cite[Theorem 5.7]{GromovLawson80}. 
    Let $x$ be an arbitrary point
    on $\partial U_j$. We have the following three cases
   according to where $x$ lies.\n   
   
   \begin{itemize}
    \item[(i)]  If $x$ is the relative interior of a facet $E$ on $\partial U_j$, then there is a unique normal direction of $\partial U_j$ at $x$.
    Let $\sigma_x$ be the geodesic segment in $\widehat{U}_j$ emanating orthogonally from $\partial U_j$ at $x$. Then $\sigma_x \times [-1,1]$ is totally geodesic in $\widehat{U}_j \times [-1,1]$. 
   The intersection of $\sigma_x \times [-1,1]$ with
   $D(U_j)$ is of the form shown in Figure~\ref{p-Bend}.   
   Let $\mu_1,\mu_2$
   be the principal curvatures of $\partial U_j$ at $x$. 
   Then at a point $p$ corresponding to the angle $\theta\in (-\pi\slash 2, \pi\slash 2)$, the principal curvatures of $D(U_j)$ will be of the form
   \begin{align*}
    \quad \  \lambda_0 = \frac{1}{\varepsilon} \cos\theta &+ O(\varepsilon), \ \ \lambda_1= \big(\mu_1+O(\varepsilon)\big)\cos\theta+ O(\varepsilon^2), \\
  & \lambda_2= \big(\mu_2+O(\varepsilon)\big)\cos\theta+ O(\varepsilon^2).
   \end{align*}   
   
 Let $\widehat{\kappa}$ be the scalar curvature function of $\widehat{U}_j$ (and $\widehat{U}_j\times [-1,1]$).  
 The by the Gauss equation, the scalar curvature $\kappa$ of
 $g_{j+1}$ at $p$ is of the form
   $$ \kappa = \widehat{\kappa} + \Big( \frac{2}{\varepsilon} H + O(1) \Big) \cos^2\theta + O(\varepsilon)$$
  where $H=\mu_1+\mu_2$ is the mean curvature of  
  $\partial U_j$ at $x$.  By our assumption
   of $g_j$, we have $\widehat{\kappa}>0$ and $H>0$. So we have $\kappa>0$ when $\varepsilon$ is sufficiently small. 
   \n
   
     \begin{figure}
        \begin{equation*}
        \vcenter{
            \hbox{
                  \mbox{$\includegraphics[width=0.9\textwidth]{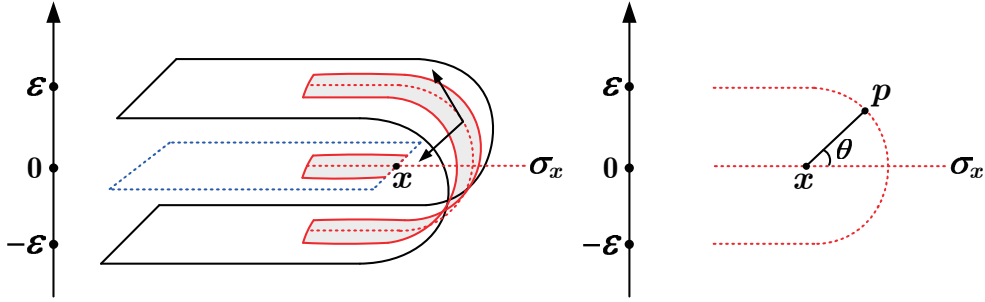}$}
                 }
           }
     \end{equation*}
   \caption{The bending region of $D(U_j)$
   parametrized by one angle
       } \label{p-Bend}
   \end{figure}

   \item[(ii)] If $x$ is the relative interior of an edge  
   $E \cap E'$ where $E$ and $E'$ are two facets on
   $\partial U_j$, the normal direction of $\partial U_j$ at $x$ in $\widehat{U}_j$ is not unique.
   Let $\sigma_x$ and $\sigma'_x$ be the geodesic segments
   orthogonal to $E$ and $E'$ at $x$, respectively.   
  Let $\Sigma^{(E,E')}_x$ be the union of the geodesic segments in $\widehat{U}_j$ orthogonal to $E\cap E'$ at $x$ which is bounded $\sigma_x$ and $\sigma'_x$. So 
  $\Sigma^{(E,E')}_x$ looks like a fan  with fan angle $\pi-\angle (E,E')_x$. 
    Note that $\angle (E,E')_x \in (0,\pi]$ by the condition (b) and (c) of $W_j$. Moreover, $x$ determines an oval-shaped patch (see Figure~\ref{p-Fan}) in the intersection of $\Sigma^{(E,E')}_x \times [-1,1]$ and $D(U_j)$.
   Let $\mu_1,\mu_2$
   be the principal curvatures of $E$ at $x$.    
   At a point $p$ corresponding to the angles
   $ \phi \in [0, \pi-\angle (E,E')_x]$ and $\theta\in (-a_{\phi}, a_{\phi})$ where $0<a_{\phi} \leq \pi\slash 2$ is determined by $\phi$,
     the principal curvatures of $D(U_j)$ will be of the form (see Figure~\ref{p-Fan}): 
   \begin{align*}
    \qquad \quad  \lambda_0 = \frac{1}{\varepsilon} \cos\phi
    \cos&\, \theta  + O(\varepsilon), \ \ \lambda_1= \big(\mu_1+O(\varepsilon)\big)\cos\phi \cos\theta + O(\varepsilon^2), \\
  & \lambda_2= \big(\mu_2+O(\varepsilon)\big)\cos\phi
  \cos\theta + O(\varepsilon^2).
   \end{align*}
    So we have  
    $$ \kappa = \widehat{\kappa} + \Big( \frac{2}{\varepsilon} H + O(1) \Big) \cos^2\phi\cos^2\theta + O(\varepsilon).$$
   By our assumption of $\widehat{\kappa}$ and $H$, we have $\kappa>0$ when $\varepsilon$ is sufficiently small.\n
        \begin{figure}
        \begin{equation*}
        \vcenter{
            \hbox{
                  \mbox{$\includegraphics[width=0.57\textwidth]{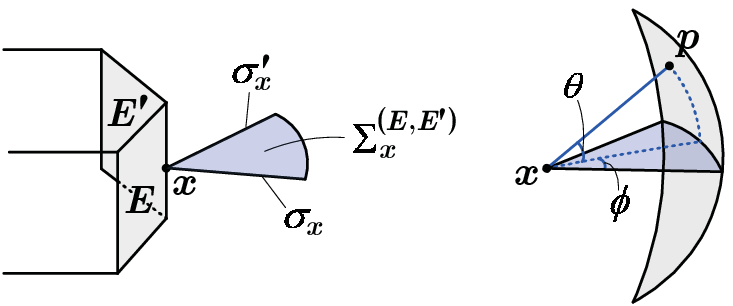}$}
                 }
           }
     \end{equation*}
   \caption{The bending region of $D(U_j)$
   parametrized by two angles
       } \label{p-Fan}
   \end{figure}
   \begin{figure}
        \begin{equation*}
        \vcenter{
            \hbox{
                  \mbox{$\includegraphics[width=0.75\textwidth]{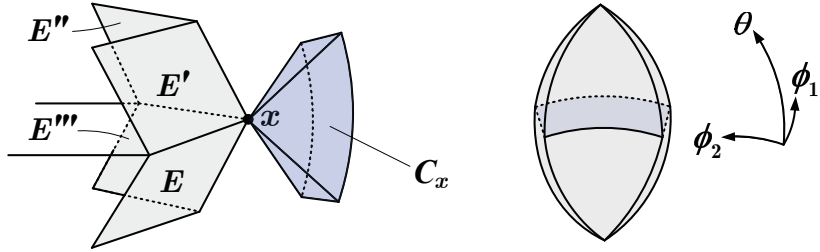}$}
                 }
           }
     \end{equation*}
   \caption{The bending region of $D(U_j)$
   parametrized by three angles
       } \label{p-Cone}
   \end{figure}

   \item[(iii)] If $x=E \cap E' \cap E''\cap E'''$ is a vertex where $E$, $E'$, $E''$ and $E'''$ are four facets 
   on $\partial U_j$, 
   the geodesic segments in $\widehat{U}_j$ emanating  from $x$ determine a cone $C_x$ which is bounded by the fans
  $\Sigma^{(E,E')}_x$, $\Sigma^{(E',E'')}_x$, $\Sigma^{(E'',E''')}_x$ and $\Sigma^{(E''',E)}_x$ (see Figure~\ref{p-Cone}). Moreover, $x$ determines a football-shaped region
  in the intersection of $C_x\times [-1,1]$ and
  $D(U_j)$ which is parametrized by three angles
  $\phi_1,\phi_2$ and $\theta$, where 
  $ \phi_1 \in [0, \pi-\angle (E,E')_x]$, $
  \phi_2 \in [0, \pi-\angle (E,E'')_x]$ and $\theta\in (-a_{\phi_1,\phi_2} , a_{\phi_1,\phi_2})$, where $0<a_{\phi_1,\phi_2}\leq \pi\slash 2$ is determined by
  $\phi_1$ and $\phi_2$.  
   Similarly to the previous cases, for a point corresponding to $\phi_1,\phi_2$ and $\theta$, we obtain
  \[\quad \kappa = \widehat{\kappa} + \Big( \frac{2}{\varepsilon} H + O(1) \Big) \cos^2\phi_1\cos^2\phi_2\cos^2\theta + O(\varepsilon).\]
   By our assumption of $\widehat{\kappa}$ and $H$, we have $\kappa>0$ when $\varepsilon$ is sufficiently small.
   \end{itemize}
   \n

  So in all cases, we have $\kappa >0$.
   By  the above discussion,
  the bending region of $D(U_j)$ can be written as $A_1\cup A_2\cup A_3$
  where $A_1$, $A_2$ and $A_3$ consist of the points parametrized by one, two and three angles, respectively. Besides, let $A_0$ be the region of $D(U_j)$ that is parallel to $U_j\times \{0\}$.
 There are some obvious creases at the intersections of  $A_0$, $A_1$, $A_2$ and $A_3$ where the metric $g_{j+1}$ does not have continuous second derivatives. But by  some small perturbations of $D(U_j)$, we can smooth out these
  creases so that the condition $\kappa>0$ still holds.
  The local model of the smoothing is given by the Cartesian product of an open subset of the quadrant $\R^2_{\geq 0}$ with the curve $\gamma$ shown in Figure~\ref{p-Smoothing} where 
  $\gamma$ is smooth except at $t=0$. 
  For every point $x\in \R^2_{\geq 0}$, we deform $\{x\}\times \gamma$ to be 
  $\{x\}\times \widetilde{\gamma}$, where $\widetilde{\gamma}$ is 
  an everywhere smooth curve that differs from $\gamma$ only in a small interval $0 < t < \delta$ for some $\delta\ll 1$ (see Figure~\ref{p-Smoothing}).
   After smoothing the creases,
  $D(W_j)$ becomes a Riemannian polyhedron
  in $(M_{j+1},g_{j+1})$.  \n
  
    \begin{figure}
        \begin{equation*}
        \vcenter{
            \hbox{
                  \mbox{$\includegraphics[width=0.72\textwidth]{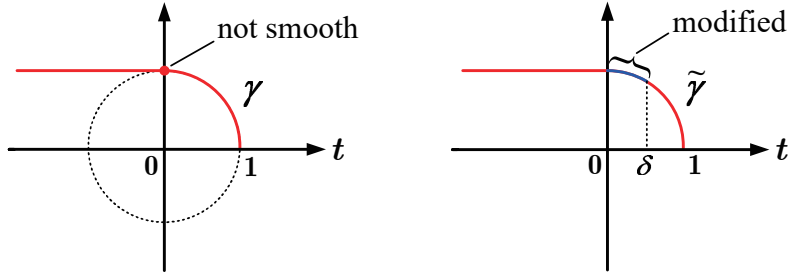}$}
                 }
           }
     \end{equation*}
   \caption{Smoothing a curve $\gamma$ near $t=0$
       } \label{p-Smoothing}
   \end{figure}

   Next, let us verify that $D(W_j)$ satisfies the conditions (b) and (c). By the definition of $\widehat{W}_j$ and $D(W_j)$, any facet $G$ of $D(W_j)$ in the bending region is of the following form:
   $$ G= \big( I_{E\cap F}\times [-1,1] \big) \cap D(W_j) $$
   where $E, F$ are two facets of $W_j$ with $E\subseteq
   \varphi_j(F^{(j)}_{j+1})$ and $F\nsubseteq \varphi_j(F^{(j)}_{j+1})$, and $I_{E\cap F}$ 
   is union of the geodesic segments in $\widehat{U}_j$ orthogonal to $E$ that emanate from points in $E\cap F$ (see Figure~\ref{p-Special-Facet}).      
     This implies:
     
     \begin{itemize}
      \item $G$ is a totally geodesic facet of $D(W_j)$.\n
    \item   
   $\angle (\widetilde{E}_{\pm \varepsilon},G) =\pi\slash 2$
   where $\widetilde{E}_{\pm \varepsilon}$ is the parallel copy of $E$ in $W_j\times \{\pm \varepsilon\}$.
   \end{itemize} 
   
    \begin{figure}
        \begin{equation*}
        \vcenter{
            \hbox{
                  \mbox{$\includegraphics[width=0.68\textwidth]{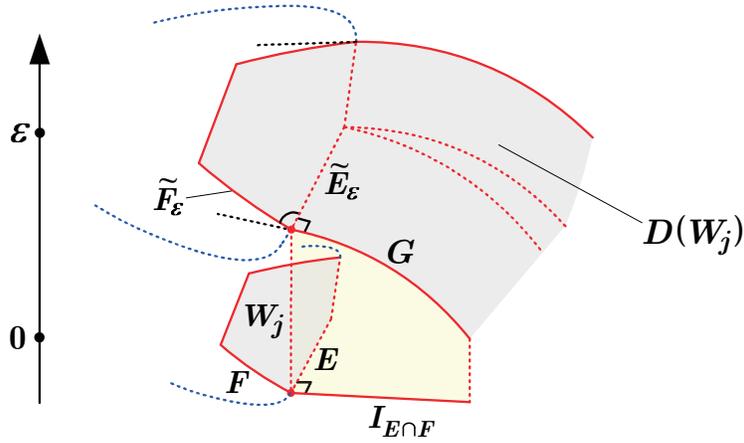}$}
                 }
           }
     \end{equation*}
   \caption{Facets of $D(W_j)$ in the bending region
       } \label{p-Special-Facet}
   \end{figure}
  
  Let $\widetilde{F}_{\pm \varepsilon}$ be the parallel copy of $F$ in $W_j\times \{\pm \varepsilon\}$.  By the condition (b) of $W_j$, the dihedral angles $\angle (\widetilde{E}_{\varepsilon},\widetilde{F}_{\varepsilon})$ and $\angle (\widetilde{E}_{- \varepsilon},\widetilde{F}_{- \varepsilon})$ are both non-obtuse.  So
   \begin{equation} \label{Equ-Dih-Angle}
     \angle (G,\widetilde{F}_{\pm \varepsilon}) \in (0,\pi].
     \end{equation}
    
   If an edge $L$ of $D(W_j)$ is not contained in the two parallel copies of $\varphi_j(F^{(j)}_{j+1})$ in $W_j\times \{\pm 1\}$,
    the dihedral angles of $D(W_j)$ at $L$ will agree with 
    the dihedral angles of the counterpart of $L$ in
    $W_j\times \{0\}$. These include all the Type-I edges 
   of $D(W_j)$.   Then since $W_j$ satisfies the condition (b), so does $D(W_j)$. \n
 
  All the Type-II edges of $D(W_j)$ that lie in the two parallel copies of $\varphi_j(F^{(j)}_{j+1})$ in $W_j\times \{\pm \varepsilon\}$ are:
   $$\big\{ G\cap \widetilde{F}_{\pm \varepsilon}\,;\, F\nsubseteq \varphi_j(F^{(j)}_{j+1}),\, G\ \text{is a facet of}\ D(W_j) \ \text{in the bending region} \big\}.$$
  It follows from~\eqref{Equ-Dih-Angle} that $D(W_j)$ 
  satisfies the condition (c) at these edges. 
   At other
   Type-II edges, $D(W_j)$ also satisfies the condition (c)  by our assumption of $W_j$.\n
   
   Moreover, the conditions (b) and (c) still hold for $D(W_j)$ after we smooth the creases of $D(U_j)$ since
   the smoothing does not change the tangent spaces 
   of the facets of $D(U_j)$ incident to the creases.\n

  In addition, $D(W_j)$ is mean curvature convex
   in $(M_{j+1},g_{j+1})$
   since its facets in the bending region are all totally geodesic while its facets outside the bending region have the same mean curvatures
   as their counterparts in $W_j\times \{0\}$.     
   \n
  So from the above arguments, we obtain a pseudo-diffeomorphism $\varphi_{j+1}$ from
  $Y^{(j+1)}$ to $W_{j+1}:=D(W_j)$ where
  $W_{j+1}$ is a mean curvature convex Riemannian polyhedron  in 
   $(M_{j+1},g_{j+1})$ with positive scalar curvature and
   $W_{j+1}$ satisfies the conditions (b) and (c).
  This finishes the induction. \n
  
  Observe that when $j=m-1$, 
  $\partial U_{m-1} = \partial W_{m-1} =\varphi_{m-1}(F^{(m-1)}_{m})$. This implies $W_{m-1}=U_{m-1}$. So by our doubling construction~\eqref{Equ-D-Wj}, the pseudo-diffeomorphism 
  $\varphi_m : Y^{(m)} = \R\mathcal{Z}_P \rightarrow W_m = D(W_{m-1})$
  is actually a diffeomorphism,
  where $W_m$ is a compact
   Riemannian $3$-manifold with positive scalar curvature.
   Then by Theorem~\ref{Thm-WuYu21}, $P$ must be combinatorially equivalent to a convex polytope that can be obtained from $\Delta^3$ by a sequence of vertex cuts.  So we finish the proof of the ``only if'' part and hence the whole theorem.
   \qed

\n
\begin{rem}
  In the proof of the ``only if'' part of Theorem~\ref{Thm-Main-1}, there is a canonical action of $H_j\cong (\Z_2)^j$ on both $Y^{(j)}$
  and $W_j$, $1\leq j \leq m$. Indeed, we can define the action of $H_j$ on $W_j$ inductively through the doubling construction $D(W_{j-1})$.
 From our definition of the pseudo-diffeomorphism $\varphi_j: Y^{(j)}\rightarrow W_j$, it is easy to see that $\varphi_j$
  is equivariant with respect to the canonical $H_j$-actions for all $1\leq j \leq m$.
\end{rem}

   \begin{cor} \label{Cor-Constant-SC}
    The statements of Theorem~\ref{Thm-Main-1} and Corollary~\ref{Cor-Main-2} still hold if we assume that the ambient
     Riemannian $3$-manifold in these two theorems has
      positive constant scalar curvature. 
  \end{cor} 
  \begin{proof}
    By the solution of the equivariant Yamabe problem in Hebey-Vaugon~\cite{HebVau93}, if $\R\mathcal{Z}_P$ has a
    $(\Z_2)^m$-invariant
    Riemannian metric $g_0$ with positive scalar curvature, then there exists a $(\Z_2)^m$-invariant Riemannian metric $\overline{g}_0$ conformal to $g_0$
    on $\R\mathcal{Z}_P$ which has positive constant scalar curvature. So
    we can prove this corollary by applying the same proof of the ``if'' part of Theorem~\ref{Thm-Main-1} to $(\R\mathcal{Z}_P,\overline{g}_0)$.
  \end{proof}
  
\n
 
  \section{Examples in higher dimensions} \label{Sec-High-Dim}

   The argument of ``doubling-smoothing'' of Riemannian polyhedra in the proof of
     Theorem~\ref{Thm-Main-1} can be generalized to any dimension $n\geq 2$ (also see~\cite[\S 2.1]{Gromov14}). 
     Indeed, the only new ingredient in the proof of the higher dimensions is: there are more types of local configurations of facets on the boundary of the polyhedra in our iterated doubling construction.      
     So     
  if an $n$-dimensional simple convex polytope $P$ in $\R^n$ can be realized
   as a mean curvature convex Riemannian polyhedron 
   with non-obtuse dihedral angles in a Riemannian $n$-manifold with positive scalar curvature, then we can construct a Riemannian metric on $\R\mathcal{Z}_P$
    with positive scalar curvature.
  Note that the iterated doubling construction induces
    a smooth structure on $\R\mathcal{Z}_P$ from
    $P$ which is equivariant with respect to the canonical $(\Z_2)^m$-action. Then since the equivariant smooth structure on
    $\R\mathcal{Z}_P$ is known to be unique up to equivariant diffeomorphisms 
    (see~\cite[Proposition 3.8]{KurMasYu15}), this smooth structure on $\R\mathcal{Z}_P$ should agree with
    the smooth structure on $\R\mathcal{Z}_P$ determined 
    by the embedding $i_{\mathcal{Z}}: \R\mathcal{Z}_P\rightarrow \R^m$ (see Section~\ref{Sec-Real-MAM}).    
   But unfortunately, we do not have the analogue
   of Theorem~\ref{Thm-WuYu21} for $\R\mathcal{Z}_P$ in higher dimensions.
   So we cannot determine all possible combinatorial types of $P$ in dimension greater than $3$.\n
   
   On the other hand, we can construct many examples of such kind of simple convex polytopes $P$ in arbitrarily high dimensions as follows.
 Let us first quote some well known results on the existence of Riemannian metrics with positive scalar curvature.

       \begin{thm}[see Gromov-Lawson~\cite{GromovLawson80-2} and Schoen-Yau~\cite{SchYau79-2}] \label{Thm-Surgery}
 Let $N$ be a closed manifold which admits a Riemannian metric with positive scalar curvature.
If $M$ is a manifold that is obtained from $N$ by surgery in codimension $\geq 3$, then $M$ also admits a Riemannian metric with positive scalar curvature.
\end{thm}
 
 Moreover, there is an equivariant version of Theorem~\ref{Thm-Surgery} for manifolds equipped with a compact Lie group action as follows.

  \begin{thm}[Theorem 11.1 in~\cite{Ber83}] \label{Thm-Surgery-Equiv}
 Let $M$ and $N$ be $G$-manifolds where $G$ is a compact Lie group.
Assume that $N$ admits an $G$-invariant metric of positive scalar curvature. If $M$
is obtained from $N$ by equivariant surgeries of codimension at least three, then $M$
admits a $G$-invariant metric of positive scalar curvature.
\end{thm}

It is shown in Bosio-Meersseman~\cite[Lemma 2.3]{BosMeer06} that up to combinatorial equivalence
   any $n$-dimensional simple convex polytope $P$ can be obtained from the $n$-simplex by a finite number of flips
   at some proper faces.
 Let $f$ be a proper face of $P$ which is combinatorially equivalent to the $k$-simplex $\Delta^k$.  
 Roughly speaking, the \emph{flip of $P$ at $f$}
   gives us a new polytope, denoted by $\mathrm{flip}_{f}(P)$, which is obtained by cutting off a small neighborhood $N(f)$ of $f$ from $P$
   and a neighborhood $N(\Delta^k)$ of
    $\Delta^k$ from $\Delta^n$, and then gluing $P\backslash N(f)$ and $\Delta^n\backslash N(\Delta^k)$ together along their cutting sections and merging the nearby facets (see~\cite{McM93,Tim99} or~\cite{BosMeer06} for the precise definition). For example, Figure~\ref{p-Flip} shows the flip of a $3$-dimensional simple convex polytope at a vertex and at an edge, respectively.
 Note that doing a flip on $P$ at a vertex $v$ is equivalent to cutting off $v$ from $P$, which will increase the number of facets by one. But whenever $\dim(f)>0$, the number of facets of $\mathrm{flip}_{f}(P)$ will be equal to that of $P$.  
In addition, a flip of $P$ at a face corresponds to
  a \emph{bistellar move} on $\partial P^*$ where $P^*$ is the dual simplicial polytope of $P$ (see~\cite{McM93,BosMeer06}).\n
  
   \begin{figure}
        \begin{equation*}
        \vcenter{
            \hbox{
                  \mbox{$\includegraphics[width=0.98\textwidth]{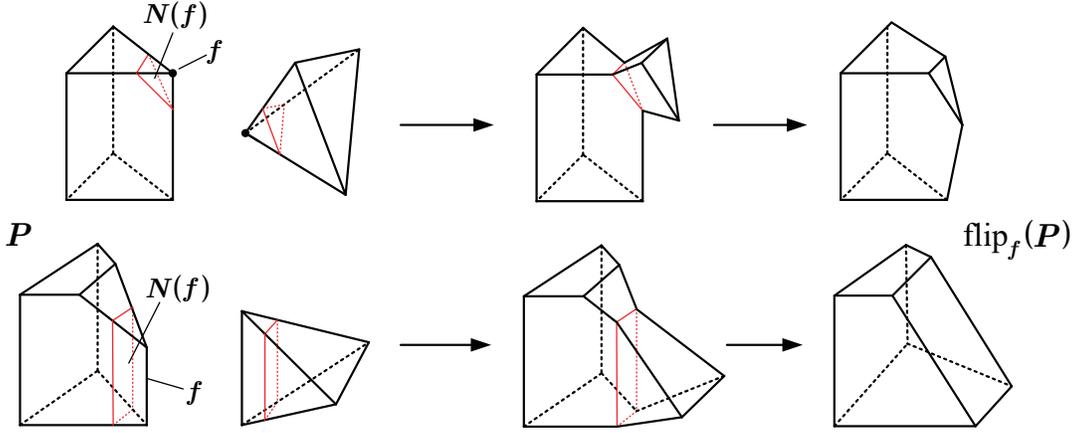}$}
                 }
           }
     \end{equation*}
   \caption{Flip of a simple convex polytope $P$ at a face $f$ } 
   \label{p-Flip}
  \end{figure}
        
  From the viewpoint of the construction of real moment-angle manifolds, a flip of $P$ at a codimension-$k$ face $f$ corresponds to an equivariant surgery on $\R\mathcal{Z}_P$ at some codimension-$k$ submanifolds.
  More specifically, it is
 a $(\Z_2)^m$-equivariant surgery on $\R\mathcal{Z}_P$ if $\dim(f)>0$  or a
 $(\Z_2)^{m+1}$-equivariant surgery on $\R\mathcal{Z}_P \times \Z_2$ if $\dim(f)=0$, where $m$ is the number of facets of $P$ (see~\cite[\S 4]{LuWangYu19}). Then it follows from~\cite[Lemma 2.3]{BosMeer06} that we can obtain
$\R\mathcal{Z}_P$ from $\R\mathcal{Z}_{\Delta^n}= S^n$
(see Example~\ref{Exam-Simplex})
by a sequence of equivariant surgeries for any $n$-dimensional simple convex polytope $P$. Moreover, the sequence of equivariant surgeries induce an (unique) equivariant smooth structure on $\R\mathcal{Z}_P$ step by step from the standard smooth structure on $S^n$.\n

  \begin{prop} \label{Prop-Examples}
     Suppose $P$ is an $n$-dimensional simple convex polytope in $\R^n$ with $n\geq 3$. If $P$ is combinatorially equivalent to a convex polytope that can be obtained from the $n$-simplex $\Delta^n$ by a sequence of flips at faces of codimension at least three, then $P$ can be realized as a right-angled totally geodesic Riemannian polyhedron in a Riemannian $n$-manifold with positive scalar curvature.
   \end{prop}
   \begin{proof}
    By our assumption of $P$,  $\R\mathcal{Z}_P$
    can be obtained from $\R\mathcal{Z}_{\Delta^n}= S^n$
    by a sequence of equivariant surgeries at some submanifolds with codimension at least three.  
  Note that the induced Riemannian metric $g_0$ on $S^n$ from $\R^{n+1}$ has positive constant scalar (sectional) curvature,
     and $g_0$ is invariant with respect to the canonical 
     action of $(\Z_2)^{n+1}$ on $S^n$ (see Example~\ref{Exam-Simplex}). Then it follows from Theorem~\ref{Thm-Surgery-Equiv} that $\R\mathcal{Z}_P$ has a $(\Z_2)^m$-invariant Riemannian metric $g_1$ with positive scalar curvature where $m$ is the number of facets of $P$.
    Finally, by applying the 
    proof of the ``if'' part of Theorem~\ref{Thm-Main-1} to $(\R\mathcal{Z}_P,g_1)$,
     we can deduce that
    $P$ is realized as a right-angled totally geodesic Riemannian polyhedron in $(\R\mathcal{Z}_P,g_1)$.
   \end{proof}
   
 To explore the generalization of Theorem~\ref{Thm-Main-1} to higher dimensions, it is natural for us to ask the following question.\n
   
 \begin{itemize}
  \item[\textbf{Question}:] Is it true that the examples given in Proposition~\ref{Prop-Examples} are exactly all the $n$-dimensional simple convex polytopes in $\R^n$ ($n\geq 3$) that can be realized as
   mean curvature convex (or totally geodesic) Riemannian polyhedra with non-obtuse dihedral angles in Riemannian $n$-manifolds with positive scalar curvature?
  \end{itemize}
 \n
 
 To answer the above question, we need to understand 
 for what simple convex polytope $P$ the manifold $\R\mathcal{Z}_P$ can admit a Riemannian metric with positive scalar curvature. It is well known that in dimension $\geq 5$, such kind of problems are related to the existence of spin structures, $\widehat{A}$-genus and $\alpha$-invariant defined by index theory (see~\cite{Stolz02}). In dimension $4$, the existence of 
 Riemannian metrics with positive scalar curvature
 implies the vanishing of Seiberg-Witten invariants.
 But the calculations of these invariants for $\R\mathcal{Z}_P$ are very difficult in general. 
 
\nn
 
  \section*{Acknowledgment}
   The author wants to thank Jiaqiang Mei, Yalong Shi, Xuezhang Chen and Yiyan Xu for some valuable discussions on Riemannian geometry.

\end{document}